\documentclass[reqno, 12pt]{amsart}
\usepackage{amssymb}
\usepackage{amsmath}
\usepackage{amsthm}

\usepackage[backref,colorlinks]{hyperref} 
\usepackage{amsrefs}  
\usepackage{enumerate}
\usepackage{enumitem}

\normalsize                              
\newtheorem{theorem}{Theorem}
\newtheorem{lemma}[theorem]{Lemma}

\newtheorem{claim}[theorem]{Claim}

\newcommand{\ff}{\mathcal{F}}

\begin{document}

	\title[An upper bound for the size of a $k$-uniform intersecting family]{An upper bound for the size of a $k$-uniform intersecting family with covering number $k$.}
	\author[A.Arman]{Andrii Arman}
	\address{Department of mathematics, University of Manitoba, Winnipeg, Manitoba R3T 2N2, Canada}
	\email{armana@myumanitoba.ca}
	\thanks{}
	
	\author[T.Retter]{Troy Retter}
	\address{Department of Mathematics and Computer Science, 
		Emory University, Atlanta, GA 30322, USA}
	\email{tretter@emory.edu}
	\thanks{}
	


	
	
	

	\begin{abstract}
		Let $r(k)$ denote the maximum number of edges in a $k$-uniform intersecting 
		family with covering number $k$. Erd\H{o}s and Lov\'asz proved that $ \lfloor k! (e-1) \rfloor \leq r(k) \leq k^k.$ Frankl, Ota, and Tokushige improved the lower bound to $r(k) \geq \left( k/2 \right)^{k-1}$, and Tuza improved the upper bound to $r(k) \leq (1-e^{-1}+o(1))k^k$. We establish that $ r(k) \leq (1 + o(1)) k^{k-1}$. 
	\end{abstract}
	\maketitle
	\section{Introduction}
	
	Let $X$ be a finite set and $k$ be a positive integer. A family of sets $\mathcal{F} \subseteq {\binom{X}{k}}$ is called a $k$-uniform hypergraph, or a $k$-uniform family. The hypergraph $\mathcal{F}$ is intersecting if all $e_1, e_2 \in \mathcal{F}$ satisfy $e_1 \cap e_2 \not = \emptyset$. A set $C \subseteq X$ is called a cover of $\mathcal{F}$ if every $f \in \mathcal{F}$ satisfies $f \cap C \not = \emptyset$. The covering number of $\ff$, denoted by $\tau(\mathcal{F})$, is define by $\tau (\mathcal{F}) := \min\{|C| : C \text{ is a cover of }\mathcal{F}\}$.
	
	Define
	\begin{align}
		r(k) := \max \{|\mathcal{F}| : \mathcal{F} \text{ is} \, k\text{-uniform, intersecting, and} \, \tau(\mathcal{F}) = k  \}, \nonumber
	\end{align}
	where no restriction is placed upon the size of the vertex set $X$.
	
	In 1975, Erd\H{o}s and Lov\'asz \cite{erdos_lovasz} proved that
	\begin{align}
		\lfloor k! (e-1) \rfloor \leq r(k) \leq k^k. \nonumber
	\end{align}
	In 1994, Tuza \cite{tuza} improved the upper bound, and in 1996, Frankl, Ota, and Tokushige \cite{frankl} improved the lower bound. It follows from these result that
	\begin{align}
		\left( \frac{k}{2} \right)^{k-1} \leq  r(k) \leq  (1-e^{-1}+o(1))k^k. \nonumber
	\end{align}
	
	Our main result is an improved upper bound. This will be established by using the following two lemmas, which will be proved in Sections \ref{section:main} and \ref{section:theorem:best}, respectively. The first lemma is based upon the degree of a vertex $x \in X$, denoted $d(x)$, which is the number of elements in $\ff$ that contain $x$.
	
	\begin{lemma} \label{theorem:main}
		Let $\ff$ be a $k$-uniform intersecting family on $X$ with covering number $k$. If $x \in X$ satisfies $d(x) \geq (\log k)k^{k-2}$, then
		\begin{align}
			|\{f \in \ff : f \not \ni x \}| = o(k^{k-1}). \nonumber
		\end{align}
	\end{lemma}
	
	The next lemma is based upon the maximum degree of a hypergraph $\ff$ on $X$, which is defined by $\Delta(\ff) := \max \{ d(x) : x \in X\}$.
	
	\begin{lemma} \label{theorem:max_deg}
		Let $\mathcal{F}$ be a $k$-uniform intersecting family on $X$ with covering number $k$. Let $\alpha \in \mathbb{R}^+$. If $\Delta(\ff) \leq | \mathcal{F} | / 40 \alpha \log k$,
		then for $k$ sufficiently large
		\begin{align}
			|\ff| \leq \max \{ 2 k^{2k/3} ,  e k^{k-\alpha} \}. \nonumber
		\end{align}
	\end{lemma}
	
	Together, these two lemmas will be used to prove our main result.
	
	\begin{theorem} \label{theorem:best}
		The function $r(k)$ satisfies
		\begin{align}
			r(k) \leq (1+o(1))k^{k-1}. \nonumber
		\end{align}
	\end{theorem}
	
	\begin{proof}
		Let $\mathcal{F}$ be a $k$-uniform intersecting family on $X$ with covering number $k$. We consider two cases.
		
		If $\Delta(\ff) \geq  (\log k)k^{k-2}$, let $x \in X$ be a vertex with $d(x) \geq  (\log k)k^{k-2}$. A simple observation (which follows from Lemma~\ref{lemma:min_d(u)}), is that any $k$-uniform intersecting family $\ff$ with covering number $k$ satisfies $\Delta(\ff) \leq k^{k-1}$. From this observation and Lemma~\ref{theorem:main},
		\begin{align}
			|\ff| \leq d(x) + |\{f \in \ff : f \not \ni x \}| \leq  k^{k-1} + o(k^{k-1}), \nonumber
		\end{align}
		as desired.
		
		In the complementary case $\Delta(\ff) <  (\log k)k^{k-2}$, we proceed by contradiction. That is, assume that $\Delta(\ff) <  (\log k)k^{k-2}$ and that $|\ff| > k^{k-1}$. For $\alpha = k / 40 \log^2 k$, we have that
		\begin{align}
			\Delta(\ff) <  (\log k)k^{k-2} \leq  \frac{|\ff|}{40 \alpha \log k}, \nonumber
		\end{align}
		and hence Lemma~\ref{theorem:max_deg} gives that for $k$ sufficiently large
		\begin{align}
			|\ff| \leq \max \left\{2k^{2k/3}, \; e k^{k-\alpha} \right\} < k^{k-1}, \nonumber
		\end{align}
		completing the proof.
	\end{proof}

	Lemmas \ref{theorem:main} and \ref{theorem:max_deg} will be established in Sections \ref{section:main} and \ref{section:theorem:best}, respectively. The next section will introduce some notation, a pair of general lemmas, and a Guesser-Chooser game upon which the proofs of Lemmas \ref{theorem:main} and \ref{theorem:max_deg} will be based. 
	
	\section{Preliminaries} \label{section:preliminaries}
	
	We will use the following notation. For $\widetilde{\ff} \subseteq \ff$ and $S \subset X$, the degree of $S$ in $\widetilde{\mathcal{F}}$, denoted by $d_{\widetilde{\ff}}(S)$, is defined by  $d_{\widetilde{\ff}}(S) := | \{f \in \widetilde{\ff} : f \supseteq S\}|$. We also take $d(S):=d_{\ff}(S)$.  For integers $i$ and $j$, let $[i]  := \{1,2, \dots ,i \}$, let $[i,j] := [j] \setminus [i-1]$, and let $(i,j] := [j] \setminus [i]$. We write $(\log k - i)$ to stand for $(\log(k) - i)$.
	
	We begin by establishing the following lemma.
	\begin{lemma} \label{lemma:shat_from_s}
		Let $\ff$ be a $k$-uniform intersecting family on $X$ with covering number $k$ and let $\widetilde{\ff} \subseteq \ff$. Let $j \in [k]$ and let $S_{j-1}$ by any subset of $X$ with size $|S_{j-1}|=j-1$. Then there exists $S_{j} = \{s_j\} \cup S_{j-1}$ with $|S_j|=j$ such that 
		\begin{align}
			d_{\widetilde{\ff}}(S_j) \geq k^{-1} \cdot d_{\widetilde{\ff}}(S_{j-1}).
			\nonumber
		\end{align}
	\end{lemma}
	
	\begin{proof}
		Since $\ff$ has covering number $k$ and $|S_{j-1}|<k$, there is an edge $f \in \ff$ such that $f \cap S_{j-1} =\emptyset$. Because $\ff$ is an intersecting family, 
		\begin{align}
			\sum_{x \in f}d_{\widetilde{\ff}}(S_{j-1}\cup{x})\geq d_{\widetilde{\ff}}(S_{j-1}). \nonumber
		\end{align}
		Therefore, for some $\widetilde{x} \in f$, we have $d_{\widetilde{\ff}}(S_{j-1}\cup {\widetilde{x}}) \geq k^{-1} \cdot d_{\widetilde{\ff}}(S_{j-1})$. It suffices to take $s_j:=\widetilde{x}$.
	\end{proof}
	
	We will also make use of the following lemma.
	
	\begin{lemma} \label{lemma:min_d(u)}
		Let $\mathcal{F}$ be a $k$-uniform intersecting family on $X$ with covering number $k$. If $U \subset X$ with $|U|=u$, then $d(U) \leq k^{k-u}$.
	\end{lemma}
	
	\begin{proof} 
		We induct on $u$. If $u=k$, then $d(U) \leq 1$.
		For $u < k$, choose $f \in \mathcal{F}$ such that $f \cap U = \emptyset$; 
		such an edge exists since $\tau(F)=k>u$.
		Making use of the fact that every edge containing $U$ must intersect $f$ and our inductive hypothesis,
		\begin{align}
			d(U) \leq \sum_{x \in f} d(U \cup\{x\}) \leq k \cdot k^{k-(u+1)} = k^{k-u}, \nonumber
		\end{align}
		completing the proof.
	\end{proof}
	For $U=\emptyset$, this yields
	\begin{align}
		r(k) \leq k^k, \label{equation:erdos_lovas}
	\end{align}
	as first proved by Erd\H{o}s and Lov\'asz in \cite{erdos_lovasz}. We now give another proof of \eqref{equation:erdos_lovas} in order to introduce some of the key ideas and notation that will be used in the proofs of Lemmas \ref{theorem:main} and \ref{theorem:max_deg}.
	
	\begin{proof}[Guesser-Chooser proof of equation \eqref{equation:erdos_lovas}.]
		We consider a game played between a Chooser and a Guesser. The game is played on a fixed hypergraph $\mathcal{F}$, which is known to both players. The Chooser selects and edge $e \in \mathcal{F}$ which is not revealed to Guesser. Guesser then ask a sequence of question $\Omega_1, \Omega_2, \dots, \Omega_k$ to gain information about the edge $e$. Each question $\Omega_i$ must have a unique answer $\omega_i \in [k]$. If Guesser can always determine the edge $e$ after asking $k$ such question, it follows that $|\mathcal{F}| \leq k^k$. Equivalently, this can be thought of as creating an injection from $\mathcal{F}$ to the set of all sequences of the form $\omega_1, \omega_2, \dots, \omega_k$ where $\omega_i \in [k]$ for all $i \in [k]$.
		
		We remark that in this game, the questions Guesser asks may depend on the answers to the previous questions, but can not depend on knowledge of the edge $e$ that is not available to Guesser.
		
		We now describe such a $k$ question strategy for Guesser. Guesser first selects an arbitrary edge $e_1 \in \mathcal{F}$ and fixes an arbitrary labeling $e_1 = \{e_1^1,e_2^1, \dots, e_k^1 \}$. Question $\Omega_1$ asks for least $\omega_1$ such that $e_{\omega_1}^1 \in e$; indeed, since $\mathcal{F}$ is a $k$-uniform intersecting family, there is a unique answer $\omega_1 \in [k]$. Hence, our first question identifies one vertex $e^1_{\omega_1} \in e$.
		
		More generally, question $\Omega_i$ is determined as follows. Guesser selects an edge $e_i \in \mathcal{F}$ that does not intersect $ \{e_{\omega_1}^1, e_{\omega_2}^2, \dots e_{\omega_{i-1}}^{i-1} \}$, which exists since $\tau(\mathcal{F}) = k$. Guesser then fixes an arbitrary labeling $e_i = \{e_1^i,e_2^i, \dots, e_k^i \}$. Question $\Omega_i$ asks for the least $\omega_i$ such that $e^i_{\omega_i} \in e$. Hence, after $k$ questions are asked, Guesser has determined
		$ e = \{e_{\omega_1}^1, e_{\omega_2}^2, \dots e_{\omega_{k}}^{k} \}$.
	\end{proof}

	\section{Proof of Lemma \ref{theorem:main}} \label{section:main}

	Let $\mathcal{F}$ be a $k$-uniform intersecting family on $X$ with $\tau(\mathcal{F}) = k$. Let $x \in X$ with $d(x) \geq (\log k) k^{k-2}$. Let
	\begin{align}
		t:= \lfloor \log k \rfloor. \nonumber
	\end{align}
	
	To show $|\{f \in \ff : f \not \ni x \}| \leq k^{k-1}$, we will make use of the Guesser-Chooser game introduced in the Guesser-Chooser proof of Equation \eqref{equation:erdos_lovas}. Chooser will select and edge $e \in \mathcal{F}$ with $e \not \ni x$ and then Guesser will ask a sequence of $k$ questions $\Omega_1, \Omega_2, \dots, \Omega_k$ that will yield corresponding answers $\omega_1, \omega_2, \dots, \omega_k$ with $\omega_i \in [k]$ for all $i \in [k]$. Unlike the previous proof, Guesser will now choose his questions so that the first $t$ answers form a non-decreasing sequence, that is
	\begin{align}
		\omega_1 \leq  \omega_2 \leq \dots \leq \omega_t. \nonumber
	\end{align}
	
	The key idea to our proof is that for $i \in [t]$, after having asked questions $\Omega_1, \Omega_2, \dots, \Omega_i$ and received answers $\omega_1,  \omega_2, \dots , \omega_i,$ Guesser will have determined
	\begin{itemize}
		\item a set $V_i \subset e$ with $|V_i|=i$,
		\item a set $U_i \subset X \setminus e$ with $|U_i|= \omega_i-1$, and
		\item a collection of edges $\mathcal{F}_{i} := \{f \in \mathcal{F}: f \supseteq U_{i} \text{ and } f \cap V_i = \emptyset\}$ with
		\begin{align} \label{eq:size_of_fi}
			|\mathcal{F}_i| \geq (\log k-i) k^{k-\omega_i}.
		\end{align}
	\end{itemize}
	
	We will say that the sets $V_i$ and $U_i$ exhibit property $\mathcal{P}_i$ if all three of the above criteria are satisfied. 
	Let $V_0 := \emptyset$, let $U_0 := \{x\}$, and let $\omega_0 := 2$. It follows that $|\mathcal{F}_0| = d(x) \geq (\log k) k^{k-2}$. Observe that Guesser knows that $V_0$ and $U_0$ exhibit property $\mathcal{P}_0$.
	\begin{claim}\label{claim:1}
		Let $i \in [t]$. Given sets $V_{i-1}$ and $U_{i-1}$ exhibiting  $\mathcal{P}_{i-1}$, Guesser can ask a question $\Omega_i$ whose answer $\omega_i$ will determine sets $V_i$ and $U_i$ exhibiting property  $\mathcal{P}_i$. Moreover, Guesser can guarantee that $\omega_i \geq \omega_{i-1}$.
	\end{claim}
	\begin{proof}
		We will specify an edge $e_i = \{ e^i_1, e^i_2, \dots, e^i_k \}$. Question $\Omega_i$ will then ask for the least $\omega_i$ such that $e^i_{\omega_i} \in e$.
		
		Fix a labeling $U_{i-1} = \{ u_1, \dots, u_{\omega_{i-1}-1} \}$. For $j \in [\omega_{i-1}-1]$, take $e_j^i := u_j$. This will ensure that $\omega_i \geq \omega_{i-1}$ as desired, since $U_{i-1} \cap e = \emptyset$.
		
		Let $S^i_{\omega_i-1} := U_{i-1}$. We now proceed recursively as follows: for $j \in [\omega_{i-1},k]$, apply Lemma~\ref{lemma:shat_from_s} to $S^i_{j-1}$ with respect to $\mathcal{F}_{i-1}$ to obtain $S^i_{j} = S^i_{j-1} \cup \{e^i_{j}\}$. For $j \in [\omega_{i-1},k]$, this yields sets $S^i_{j}$ with
		\begin{align}\label{eq:d1}
			d_{\ff_{i-1}} (S^i_j) \geq  k^{-j+\omega_{i-1}-1} d_{\ff_{i-1}} (S^i_{\omega_i-1}) = k^{-j+\omega_{i-1}-1}| \mathcal{F}_{i-1}|. 
		\end{align}
		By~\eqref{eq:size_of_fi},
		\begin{align}\label{eq:d2}
			k^{-j+\omega_{i-1}-1}| \mathcal{F}_{i-1}| \geq  (\log k - i+1) k^{k-j-1} .
		\end{align}
		It follows from \eqref{eq:d1} and \eqref{eq:d2} that for $j \in [\omega_{i-1},k]$,
		\begin{align}\label{eq:d3}
			d_{\ff_{i-1}} (S^i_j) \geq  (\log k -i +1) k^{k-j-1}.
		\end{align}
		Now  making use of $i \leq t$, from \eqref{eq:d3} we have that $d_{\ff_{i-1}} (S^i_k) > 0$. From the definition of $\mathcal{F}_{i-1}$, it now follows that $S^i_{k} \cap V_{i-1} = \emptyset$. Hence, $e_j^i \not \in V_{i-1}$ for all $j \in [k]$.
		
		Having completed our construction of $e_i$, we now consider the answer $\omega_i$ to question $\Omega_i$. Define
		\begin{align}
			V_i := V_{i-1} \cup \{e^{i}_{\omega_i} \} \quad \quad \text{and} \quad \quad U_i :=\{ e^i_1, e^i_2, \dots, e^{i}_{\omega_i-1} \}. \nonumber
		\end{align}
		Observe that $\mathcal{F}_i$ is precisely the edges in $\mathcal{F}_{i-1}$ that contain $U_i=S^i_{\omega_i-1}$ and do not contain $e^{i}_{\omega_i}$. Making use of \eqref{eq:d3} and Lemma~\ref{lemma:min_d(u)},
		\begin{align}
			|\mathcal{F}_i| &\geq d_{\ff_{i-1}} (S^i_{\omega_i-1}) - d_{\ff}(S^i_{\omega_i} \cup \{e^{i}_{\omega_i}\}) \nonumber \\
			&\geq (\log k - i +1 )k^{k-\omega_i} - k^{k-\omega_i} \nonumber \\
			&=  (\log k-i)k^{k-\omega_i}. \nonumber
		\end{align}
		Thus, we have shown that $V_i$ and $U_i$ exhibit property $\mathcal{P}_i$.
	\end{proof}
	
	It follows from Claim \ref{claim:1} that Guesser may ask questions, $\Omega_1, \Omega_2 ,\dots,  \Omega_t$ that necessitate a non-decreasing sequence of answers $\omega_1,\omega_2 ,\dots,\omega_t$. Moreover, after asking these questions, Guesser will have determined $V_t \subset e$ with $|V_t|=t$. For the remaining $k-t$ questions, Guesser will no longer ask questions that necessitate a non-decreasing sequence.
	
	\begin{claim}\label{claim:2}
		Let $i \in (t,k]$. Given a set $V_{i-1} \subseteq e$ with $|V_{i-1}|=i-1$, Guesser can ask a question $\Omega_i$ whose answer $\omega_i$ will allow Guesser to determine a set $V_i \subseteq e$  with $|V_{i}|=i$.
	\end{claim}
	\begin{proof}
		Let $e_i$ be any edge not covered by $V_{i-1}$; such an edge exists since $\tau(\mathcal{F}) > i-1$. Arbitrarily label $e_i = \{ e^i_1, e^i_2, \dots, e^i_k \}$. Question $\Omega_i$ asks for the least $\omega_i$ such that $e^i_{\omega_i} \in e$. Let $V_{i}:= V_{i-1} \cup \{e^i_{\omega_i}\}$.
	\end{proof}
	
	Hence, after $k$ questions are asked, Guesser will have determined $e = V_k$. Since the first $t$ answers are non-decreasing and the number of non-decreasing sequences in $[k]^t$ is ${ \binom{k+t-1}{t} }$, this gives that
	\begin{align}
		|\{f \in \ff : f \not \ni x \}| &\leq  { \binom{k+t-1}{t}} k^{k-t} \nonumber \\
		&\leq \left(\frac{e(k+t-1)}{t}\right)^t k^{k-t}  \nonumber \\
		&\leq \left( \frac{e}{t} \left( 1+\frac{t-1}{k} \right) \right)^t k^k \nonumber \\
		&\leq \left( \frac{2e}{t} \right)^t k^k \leq  \left( \frac{2e}{\log k -1} \right)^{\log k - 1} k^k  \leq k^{k- (1-o(1))\log \log k } = o(k^{k-1}). \nonumber
	\end{align} 
	This completes the proof of Lemma~\ref{theorem:main}.
	
	\section{Proof of Lemma \ref{theorem:max_deg}}\label{section:theorem:best}
	
	Let $\mathcal{F}$ be a $k$-uniform intersecting family on $X$ with $\tau(\mathcal{F}) = k$. Suppose that $\Delta(\ff) \leq |\ff| / 40 \alpha \log k$. To prove Lemma \ref{theorem:max_deg}, it suffices to prove that if $|\ff| > 2k^{2k/3}$, then $|\ff| \leq ek^{k-\alpha}$. Hence, we assume that  $|\ff| > 2k^{2k/3}$.
	
	Let 
	\begin{align}
		t := 20 \lfloor \alpha \log k \rfloor.
	\end{align}
	
	As in the proofs of Equation~\eqref{equation:erdos_lovas} and Theorem~\ref{theorem:main}, we will make use of the Guesser and Chooser game. As before, Chooser will select and edge $e \in \mathcal{F}$ and then Guesser will ask a sequence of $k$ questions $\Omega_1, \Omega_2, \dots, \Omega_k$ that will yield corresponding answers $\omega_1, \omega_2, \dots, \omega_k$ with $\omega_i \in [k]$ for all $i \in [k]$. Unlike the previous two proofs, Guesser will now choose his questions so that
	\begin{align} 
		\omega_i >  2k/3 \implies  \omega_{i+1} >  k/3 \text{\quad for all odd $i \in [t]$}. \label{equation:new_condition}
	\end{align}
	
	Let $V_0 := \emptyset$. The following claim establishes that Guesser can ask his first $t$ questions so that \eqref{equation:new_condition} is satisfied.
	
	\begin{claim}\label{claim:new_1}
		Let $i \in [t]$ be an odd number. Given a set $V_{i-1} \subset e$ with $|V_{i-1}|=i-1$, Guesser can ask a pair of questions question $\Omega_{i}$ and $\Omega_{i+1}$ whose answers will determine a set $V_{i+1} \subset e$ with $|V_{i+1}|=i+1$. Moreover, these questions can be asked so that $\omega_{i} > 2k/3$ implies that $\omega_{i+1} > k/3$.
	\end{claim}

	\begin{proof}
		Let $i \in [t]$ be an odd number and $V_{i-1} \subset e$ with $|V_{i-1}|=i-1$. Let 
		$$\mathcal{F}_{i-1} := \{f \in \ff : f \cap V_{i-1} = \emptyset \}.$$ 
		It follows that
		$$|\mathcal{F}_{i-1}| \geq |\ff| - \Delta(\ff) \cdot  (i-1) \geq |\ff|/2 \geq k^{2k/3}.$$
		
		We now construct a testing edge $e_{i} = \{ e^{i}_1, e^{i}_2, \dots, e^{i}_k \}$. We begin by specifying the first $\lfloor k/3 \rfloor$ vertices in $e_{i}$ as follows.  Let $S^{i}_0 := \emptyset$. We now proceed recursively: for $j \in [ \lfloor k/3 \rfloor ]$, apply Lemma~\ref{lemma:shat_from_s} to $S^{i}_{j-1}$ with respect to $\mathcal{F}_{i-1}$ to obtain $S^{i}_{j} = S^{i}_{j-1} \cup \{e^{i}_{j}\}$. This yields sets $S^{i}_{j}$ with
		\begin{align}\label{eq:new_d1}
			d_{\ff_{i-1}} (S^{i}_j) \geq k^{2k/3 - j}. 
		\end{align}
		
		Having specified the first $\lfloor k/3 \rfloor$ vertices in $e_{i+1}$, we will now work to specify the remaining vertices. To this end, let $D_{\ff_{i-1}}(S^i_{\lfloor k/3 \rfloor}) := \{f \in \ff_{i-1} : f \supseteq S^i_{ \lfloor k/3 \rfloor }\}$. Define 
		\begin{align}
			P_i:= \{x \in X \setminus S^i_{\lfloor k/3 \rfloor} : x \in f \text{ for all } f \in D_{\ff_{i-1}}(S^{i}_{ \lfloor k/3 \rfloor }) \}. \label{def:pi}
		\end{align}
		It follows from \eqref{eq:new_d1}, \eqref{def:pi}, and Lemma \ref{lemma:min_d(u)} that
		\begin{align}
			k^{2k/3 - \lfloor k/3 \rfloor } \leq d_{\ff_{i-1}}(S^{i}_{ \lfloor k/3 \rfloor }) =  d_{\ff_{i-1}}(S^i_{ \lfloor k/3 \rfloor} \cup P_i) \leq k^{k- \lfloor k/3 \rfloor -|P_i|}. \label{equation:size_of_p}
		\end{align}
		The inequality in \eqref{equation:size_of_p} establishes that $|P_i| \leq k/3$. We now take $e_{i+1}$ to be any edge in $D_{\ff_{i-1}}(S^{i}_{\lfloor k/3 \rfloor})$; such an edge is guaranteed to exists since the set in non-empty by \eqref{equation:size_of_p}. Label  
		$$e_{i} = \{ e^{i}_1, e^{i}_2, \dots, e^{i}_k \}$$
		so that $\{ e^{i}_1, e^{i}_2, \dots, e^{i}_{ \lfloor k/3 \rfloor } \} = S^i_{ \lfloor k/3 \rfloor }$ and $P_i \subseteq  \{ e^{i}_{ \lfloor k/3 \rfloor +1}, e^{i}_{ \lfloor k/3 \rfloor+2}, \dots, e^{i}_{ \lfloor 2k/3 \rfloor} \}$.
		
		The question $\Omega_{i}$ asks for the least integer $\omega_{i}$ such that $e^{i}_{\omega_{i}} \in e$. Let $V_{i} := V_{i-1} \cup \{ e^{i}_{\omega_{i}} \}$. We now consider two cases depending upon the answer $\omega_i$.
		
		If $\omega_{i} > 2k/3$, then Guesser must ensure that the answer $\omega_{i+1}$ to the next question will satisfy $\omega_{i+1} \geq k/3$. Observe that since $\omega_{i} > 2k/3$, it follows that $S^{i}_{\lfloor k/3 \rfloor } \cap e = \emptyset$. Also, since $\omega_{i} > 2k/3$, we have $e^{i}_{\omega_{i}} \not \in P_i$. Hence, by the definition of $P_i$ (see \eqref{def:pi}), there exists an edge $e_{i+1} \in D_{\ff_{i-1}}(S^i_{\lfloor k/3 \rfloor })$ with $e_{i+1} \not \ni  e^{i}_{\omega_{i}}$. Label the vertices of this edge
		$$e_{i+1} = \{ e^{i+1}_1, e^{i+1}_2, \dots, e^{i+1}_k \}$$ 
		so that $\{ e^{i+1}_1, e^{i+1}_2, \dots, e^{i+1}_{\lfloor k/3 \rfloor} \} = S^i_{ \lfloor k/3 \rfloor}$. It follows that $e_{i+1} \cap V_{i} = \emptyset$. The answer to question $\Omega_{i+1}$ (based upon the testing edge $e_{i+1}$) will identify a new vertex in $e$ and necessitate an answer $\omega_{i+1} \geq k/3$.
		
		In the complementary case $\omega_{i+1} \leq 2k/3$, the question $\Omega_{i+1}$ must identify a new vertex in $e$ and the answer $\omega_{i+1}$ can be any integer in $[k]$. To accomplish this, Guesser takes the testing edge $e_{i+1}$ to be any any edge that does not intersect $V_{i}$; such an edge exists since $\tau(\ff) = k$.
		
	\end{proof}
	
	By Claim \ref{claim:1}, Guesser may ask questions, $\Omega_1, \Omega_2 ,\dots,  \Omega_t$ that necessitate a sequence of answers $\omega_1,\omega_2 ,\dots,\omega_t$ satisfying property \eqref{equation:new_condition}. Moreover, after asking these questions, Guesser will have determined $V_t \subset e$ with $|V_t|=t$. For the remaining $k-t$ questions, Guesser will only require that each answer is a numbers in $[k]$ that identifies a new vertex in $e$. This is possible by Claim \ref{claim:2}.
	
	Hence, after $k$ questions are asked, Guesser will have determined the edge $e$ selected by Chooser. It follows that the size of $|\ff|$ is bounded above by the number of sequence  $\omega_1, \omega_2, \dots, \omega_k \in  [k]^k$ that satisfy property \eqref{equation:new_condition}. Because the number of ways to select a pair $\omega_i,\omega_{i+1}  \in [k]$ with the condition in \eqref{equation:new_condition} is less than
	$$k^2 - (k/3-2)(k/3-2) = (8/9)k^2 + 4 k/3 - 4 < e^{-1/10} k^2$$
	for $k$ sufficiently large, it follows that
	$$|\ff| \leq \left( e^{-1/10} k^2 \right)^{t/2}k^{k-t} = e^{-t/20}k^k  \leq e k^{k- \alpha}$$
	for $k$ sufficiently large. This completes the proof of Lemma~\ref{theorem:max_deg}.

	\section{Acknowledgements}
	
	We would like to thank Peter Frankl and Vojt\v{e}ch R\"odl for sharing this problem with us.

\end{document}